\documentclass{article}
\usepackage{a4wide}
\usepackage{graphicx}
\usepackage{amsmath}
\usepackage{amssymb}
\usepackage{amsthm}
\usepackage{enumerate}

\newtheorem{theorem}{Theorem}[section]
\newtheorem{definition}[theorem]{Definition}
\newtheorem{proposition}[theorem]{Proposition}
\newtheorem{lemma}[theorem]{Lemma}
\newtheorem{claim}[theorem]{Claim}

\newtheorem{conjecture}[theorem]{Conjecture}

\newcommand{\dRainbow}[1]{d_{R}(#1)}
\begin{document}

\title{Rainbow matchings and rainbow connectedness}

\author{\large{Alexey Pokrovskiy} 
\\
\\ Methods for Discrete Structures,
\\ {Freie Universit\"at,} 
\\ Berlin, Germany.
\\ {Email: \texttt{alja123@gmail.com}}
}

\maketitle

\begin{abstract}
 Aharoni and Berger conjectured that every bipartite graph which is the union of $n$ matchings of size $n+1$ contains a rainbow matching of size $n$. This conjecture is a generalization of several old conjectures of Ryser, Brualdi, and Stein about transversals in Latin squares. There have been many recent partial results about the Aharoni-Berger Conjecture. In the case when the matchings are much larger than $n+1$, the best bound is currently due to Clemens and Ehrenm\"uller who proved the conjecture when the matchings are of size at least $3n/2+o(n)$. When the matchings are all edge-disjoint and perfect, then the best result follows from a theorem of H\"aggkvist and Johansson which implies the conjecture when the matchings have size at least $n+o(n)$.
 
In this paper we show that the conjecture is true when the matchings have size $n+o(n)$ and are all edge-disjoint (but not necessarily perfect). We also give an alternative argument to prove the conjecture when the matchings have size at least $\phi n+o(n)$ where $\phi\approx 1.618$ is the Golden Ratio.

Our proofs involve studying connectedness in coloured, directed graphs. The notion of connectedness that we introduce is new, and perhaps of independent interest.
\end{abstract}

\section{Introduction}
A Latin square of order $n$ is an $n\times n$ array filled with $n$ different symbols, where no symbol appears in the same row or column more than once. Latin squares arise in different branches of mathematics such as algebra (where Latin squares are exactly the multiplication tables of quasigroups) and  experimental design (where they give rise to designs called Latin square designs). They also occur in  recreational mathematics---for example completed Sudoku puzzles are Latin squares. %\textbf{BETTER EXAMPLES}

In this paper we will look for \emph{transversals} in Latin squares---a transversal in a Latin square of order $n$ is a set of $n$ entries such that no two entries are in the same row, same column, or have the same symbol. One reason transversals in Latin squares are interesting is that a Latin square has an orthogonal mate if, and only if, it has a decomposition into disjoint transversals. See \cite{WanlessSurvey} for a survey about transversals in Latin squares.
It is easy to see that not every Latin square has a transversal (for example the unique $2\times 2$ Latin square has no transversal), however perhaps every Latin square contains a large \emph{partial transversal} (a partial transversal of size $m$ is a  set of $m$ entries such that no two entries are in the same row, same column, or have the same symbol)? 

There are several closely related, old, and difficult conjectures which guarantee large partial transversals in Latin squares. The first of these is a conjecture of Ryser that every Latin square of odd order contains a transversal \cite{Ryser}.  Brualdi  conjectured that every Latin square contains a partial transversal of size $n-1$ (see \cite{Brualdi}). Stein independently made the stronger conjecture that every $n\times n$ array filled with $n$ symbols, each appearing exactly $n$ times contains a partial transversal of size $n-1$ \cite{Stein}. Because of the similarity of the above two conjectures, the following is often referred to as ``the Brualdi-Stein Conjecture''.%(and sometimes  the ``the Ryser-Brualdi-Stein Conjecture'').
\begin{conjecture}[Brualdi and Stein, \cite{Brualdi, Stein}]\label{BrualdiStein}
Every Latin square contains a partial transversal of size $n-1$.
\end{conjecture}

There have been many partial results about this conjecture. 
%It is known that the multiplication table of any abelian group of odd order has a transversal (so Ryser's conjecture holds in this case). 
%Erd\H{o}s and Spencer\cite{ErdosSpencer} showed that  any $n\times n$ array in which no entry occurs more than $(n − 1)/16$ times has a transversal.
It is  known that every Latin square has a  partial transversal of size $n-o(n)$---Woolbright \cite{Woolbright} and independently Brower, de Vries and Wieringa \cite{Brower} proved that ever Latin square contains a partial transversal of size $n-\sqrt n$. This has been improved by Hatami and Schor \cite{HatamiSchor} to $n-O(\log^2 n)$. A  remarkable result of H\"aggkvist and Johansson shows that if we consider $(1-\epsilon)n\times n$ Latin rectangles rather than Latin squares, then it is possible to decompose all the entries into disjoint transversals (for $m\leq n$ a $m\times n$ Latin rectangle is an $m\times n$ array of $n$ symbols where no symbol appears in the same row or column more than once. A transversal in a Latin rectangle is a set of $m$ entries no two of which are in the same row, column, or have the same symbol).
\begin{theorem}[H\"aggkvist and Johansson, \cite{HaggkvistJohansson}]\label{HaggvistTheorem}
For every $\epsilon$, there is an $m_0=m_0(\epsilon)$ such that the following holds.
For every $n\geq (1+\epsilon)m\geq m_0$, every $m\times n$ Latin rectangle can be decomposed into disjoint transversals.
\end{theorem}
This theorem is proved by a probabilistic argument, using a ``random greedy process'' to construct the transversals. The above theorem gives yet another proof that every sufficiently large $n\times n$ Latin square has a partial transversal of size $n-o(n)$---indeed if we remove $\epsilon n$ rows of a Latin square we obtain a Latin rectangle to which Theorem~\ref{HaggvistTheorem} can be applied.

In this paper we will look at a strengthening of Conjecture~\ref{BrualdiStein}. The strengthening we'll look at is a conjecture due to Aharoni and Berger which takes place in a more general setting than Latin squares---namely coloured bipartite graphs. 
To see how the two settings are related, notice that there is a one-to-one correspondence between $n\times n$ Latin squares and proper edge-colourings of $K_{n,n}$ with $n$ colours---indeed to a Latin square $S$ we associate the colouring of $K_{n,n}$ with vertex set $\{x_1, \dots, x_n, y_1, \dots, y_n\}$ where the edge between  $x_i$  and  $y_j$ receives colour $S_{i,j}$. It is easy to see that in this setting transversals in $S$ correspond to perfect rainbow matchings in $K_{n,n}$ (a matching is \emph{rainbow} if all its edges have different colours). Thus Conjecture~\ref{BrualdiStein} is equivalent to the statement that ``in any proper $n$-colouring of $K_{n,n}$, there is a rainbow matching of size $n-1$''.

One could ask whether a large rainbow matching exists in more general bipartite graphs.
Aharoni and Berger posed the following conjecture, which generalises Conjecture~\ref{BrualdiStein}.
\begin{conjecture}[Aharoni and Berger, \cite{AharoniBerger}]\label{ConjectureAharoni}
Let $G$ be a bipartite graph consisting of $n$  matchings, each with at least $n+1$ edges. Then $G$ contains a rainbow matching with $n$ edges.
\end{conjecture}
In the above conjecture we think of the $n$ matchings forming $G$ as having different colours, and so ``rainbow matching'' means a matching containing one edge from each matching in $G$. It is worth noting that the above conjecture does not require the matchings in $G$ to be disjoint i.e. it is about bipartite multigraphs rather than simple graphs.
This  conjecture was first posed in a different form in~\cite{AharoniBerger} as a conjecture about matchings in tripartite hypergraphs (Conjecture 2.4 in \cite{AharoniBerger}). It was first stated as a conjecture about rainbow matchings in~\cite{AharoniCharbitHoward}. 

The above conjecture has attracted a lot of attention recently, and there are many partial results. Just like in Conjecture~\ref{BrualdiStein}, one natural way of attacking Conjecture~\ref{ConjectureAharoni} is to prove approximate versions of it. As observed by Barat, Gy\'arf\'as, and S\'ark\"ozy \cite{Barat}, the arguments that Woolbright, Brower, de Vries, and Wieringa used to find  partial transversals of size size $n-\sqrt n$ in Latin squares actually generalise to bipartite graphs to give the following.
\begin{theorem}[Woolbright, \cite{Woolbright}; Brower, de Vries, and Wieringa, \cite{Brower}; Barat, Gy\'arf\'as, and S\'ark\"ozy, \cite{Barat}]\label{Woolbright}
Let $G$ be a bipartite graph consisting of $n$ matchings, each with at least $n$ edges. Then $G$ contains a rainbow matching with $n-\sqrt{n}$ edges.
\end{theorem}
Barat, Gy\'arf\'as, and S\'ark\"ozy actually proved something a bit more precise in \cite{Barat}---for every $k$, they gave an upper bound on the number of matchings of size $n$ needed to find a rainbow matching of size $n-k$.

Another approximate version of Conjecture~\ref{ConjectureAharoni} comes from Theorem~\ref{HaggvistTheorem}. It is easy to see that Theorem~\ref{HaggvistTheorem} is equivalent to the following ``let $G$ be a bipartite graph consisting of $n$ edge-disjoint perfect matchings, each with at least $n+o(n)$ edges. Then $G$ can be decomposed into disjoint rainbow matchings of size $n$'' (to see that this is equivalent to Theorem~\ref{HaggvistTheorem}, associate an $m$-edge-coloured bipartite graph with any $m\times n$ Latin rectangle by placing a colour $k$ edge between $i$ and $j$ whenever $(k,i)$ has symbol $j$ in the rectangle).

The main result of this paper is an approximate version of Conjecture~\ref{ConjectureAharoni} in the case when the matchings in $G$ are disjoint, but not necessarily perfect.
\begin{theorem}\label{MainTheorem}
For all $\epsilon>0$, there exists an $N=N(\epsilon)=10^{20}\epsilon^{-16\epsilon^{-1}}$ such that the following holds.
Let $G$ be a bipartite graph consisting on $n\geq N$ edge-disjoint matchings, each with at least $(1+\epsilon)n$ edges. Then $G$ contains a rainbow matching with $n$ edges.
\end{theorem}
Unlike the proof of Theorem~\ref{HaggvistTheorem} which can be used to give a randomised process to find a rainbow matching, the proof of Theorem~\ref{MainTheorem} is algorithmic i.e. the matching guaranteed by Theorem~\ref{MainTheorem} can be found in polynomial time.

Another very natural approach to Conjecture~\ref{ConjectureAharoni} is to to prove it when the matchings have size much larger than $n+1$. When the matchings have size $2n$, then the result becomes trivial.
\begin{lemma}\label{GreedyLemma}
Let $G$ be a bipartite graph consisting of $n$ matchings, each with at least $2n$ edges. Then $G$ contains a rainbow matching with $n$ edges.
\end{lemma}
This lemma is proved by greedily choosing  disjoint edges of different colours. We can always choose $n$ edges this way, since each colour class has $2n$ edges (one of which must be disjoint from previously chosen edges).

There have been several improvements to the $2n$ bound in Lemma~\ref{GreedyLemma}. Aharoni, Charbit, and Howard \cite{AharoniCharbitHoward} proved that matchings of size $7n/4$ are sufficient to guarantee a rainbow matching of size $n$. Kotlar and Ziv \cite{KotlarZiv} improved this to $5n/3$.  Clemens and Ehrenm\"uller \cite{DennisJulia} further improved the constant to $3n/2$ which is currently the best known bound.
\begin{theorem}[Clemens and Ehrenm\"uller, \cite{DennisJulia}]
Let $G$ be a bipartite graph consisting of $n$ matchings, each with at least $3n/2+o(n)$ edges. Then $G$ contains a rainbow matching with $n$ edges.
\end{theorem}

Though we won't improve this theorem, we give an alternative proof which gives a weaker bound of $\phi n$ where $\phi\approx 1.618$ is the Golden Ratio.
\begin{theorem}\label{GoldenRatioTheorem}
Let $G$ be a bipartite graph consisting of $n$ matchings, each with at least $\phi n + 20n/\log n$ edges. Then $G$ contains a rainbow matching with $n$ edges.
\end{theorem}

We'll spend the rest of this section informally discussing the methods which we use to prove Theorems~\ref{MainTheorem} and~\ref{GoldenRatioTheorem}. The basic idea is to introduce an auxiliary coloured directed graph and then apply some simple lemmas about directed graphs. The results and concepts about coloured graphs which we use are perhaps of independent interest. These results are all gathered in Section~\ref{SectionConnectedness}. The key idea in the proof of Theorem~\ref{MainTheorem} seems to be a new notion of connectivity of coloured graphs.
\begin{definition}\label{DefinitionConnectivity}
An edge-coloured graph $G$ is said to be  rainbow $k$-edge-connected if for any set of at most $k-1$ colours $S$ and any pair of vertices $u$ and $v$, there is a rainbow $u$ to $v$ path whose edges have no colours from $S$.
\end{definition}
The above definition differs from usual notions of connectivity, since generally the avoided set $S$ is a set of \emph{edges} rather than colours. 
As we shall see, in some ways Definition~\ref{DefinitionConnectivity} is perhaps \emph{too strong}. In particular the natural analogue of Menger's Theorem for  rainbow $k$-edge-connected graphs fails (see Section~\ref{SectionConclusion}). Nevertheless,  rainbow $k$-edge-connected graphs turn out to be very useful for studying rainbow matchings in bipartite graphs. It would be interesting to know whether any statements about edge-connectivity have natural generalizations to  rainbow $k$-edge-connected graphs.

The structure of this paper is as follows. In Section~\ref{SectionConnectedness} we introduce variations on Definition~\ref{DefinitionConnectivity} and prove a number of lemmas about coloured and directed graphs.
In Section~\ref{SectionMainTheorem} we prove Theorem~\ref{MainTheorem}.
In Section~\ref{SectionGoldenRatio} we prove Theorem~\ref{GoldenRatioTheorem}.
In Section~\ref{SectionConclusion} we make some concluding remarks about the techniques used in this paper. For all standard notation we follow~\cite{BollobasModernGraphTheory}.

\section{Paths in directed and coloured graphs}\label{SectionConnectedness}
In this section we prove results about paths in various types of directed graphs. All graphs in this section have no multiple edges, although we allow the same edge to appear twice in opposite directions.
In directed graphs, ``path'' will always mean a sequence of vertices $x_1, \dots, x_k$ such that $x_ix_{i+1}$ is a directed edge for $i=1, \dots, k-1$. We will use additive notation for concatenating paths---for two paths $P=p_1 \dots p_i$ and $Q=q_1 \dots q_j$, $P + Q$ denotes the path with vertex sequence $p_1\dots p_i q_1\dots q_j$.
Recall that $N^+(v)$ denotes the out-neighbourhood of a vertex $v$ i.e. the set of vertices $x$ for which $vx$ is an edge. We will sometimes have two graphs $G$ and $H$ on the same vertex set. In this case $N^+_G(v)$ and $N^+_H(v)$ denote the out-neighbourhoods of $v$ in $G$ and $H$ respectively. Similarly $d_G(u,v)$ and $d_H(u,v)$ denote the lengths of the shortest part from $u$ to $v$ in $G$ and $H$ respectively.

We will look at coloured graphs. An edge-colouring of a graph is an arbitrary assignment of colours to the edges of a graph. A total colouring is an arbitrary assignment of colours to both the vertices and edges of a graph. For any coloured graph we denote by $c(v)$ and $c(uv)$ the colour assigned to a vertex or edge respectively.

An edge-colouring is out-proper if for any vertex $v$, the outgoing edges from $v$ all have different colours. Similarly an edge-colouring is in-proper if for any vertex $v$, the ingoing edges from $v$ all have different colours. We say that an edge colouring is proper if it is both in and out-proper (notice that by this definition it is possible to have two edges with the same colour at a vertex $v$---as long as one of the edges is oriented away from $v$ and one is oriented towards $v$).
A total colouring is proper if the underlying edge colouring and vertex colourings are proper and the colour of any vertex is different from the colour of any edge containing it.
A totally coloured graph is \emph{rainbow} if all its vertices and edges have different colours. 
For two vertices $u$ and $v$ in a coloured graph, $\dRainbow{u,v}$ denotes the length of the shortest rainbow path from $u$ to $v$.
We say that a graph has \emph{rainbow vertex set} if all its vertices have different colours.

This section will mostly be about finding highly connected subsets in directed graphs. The following is the notion of connectivity that we will use.
\begin{definition}\label{DefinitionUncolouredkdCon}
 Let $A$ be a set of vertices in a digraph $D$. We say that $A$ is $(k,d)${-connected} in $D$ if, for any set of vertices $S\subseteq V(D)$ with $|S|\leq k-1$ and any vertices $x,y \in A\setminus S$, there is an $x$ to $y$ path of length $\leq d$ in $D$ avoiding $S$. 
\end{definition}
Notice that a directed graph $D$ is strongly $k$-connected if, and only if, $V(D)$ is $(k, \infty)$-connected in $D$. Also notice that it is possible for a subset $A\subseteq V(D)$ to be highly connected without the induced subgraph $D[A]$ being highly connected---indeed if $D$ is a bipartite graph with classes $X$ and $Y$ where all edges between $X$ and $Y$ are present in both directions, then $X$ is a $(|Y|, 2)$-connected subset of $D$, although the induced subgraph on $X$ has no edges.

We will also need a generalization this notion of connectivity to coloured graphs
\begin{definition}\label{DefinitionColouredkdCon}
 Let $A$ be a set of vertices in a coloured digraph $D$. We say that $A$ is $(k,d)${-connected} in $D$ if, for any set of at most $k-1$ colours $S$ and any vertices $x,y \in A$, there is a rainbow $x$ to $y$ path of length $\leq d$ in $D$ internally avoiding colours in $S$. 
\end{definition}
Notice that in the above definition, we did not specify whether the colouring was a  edge colouring, vertex colouring, or total colouring. The definition makes sense in all three cases. For edge colourings a path $P$ ``internally avoiding colours in $S$'' means $P$ not having edges having colours in $S$. For vertex colourings a path $P$ ``internally avoiding colours in $S$'' means $P$ not having vertices having colours in $S$ (except possibly for the vertices $x$ and $y$). For total colourings a path $P$ ``internally avoiding colours in $S$'' means $P$ having no edges or vertices with colours in $S$ (except possibly for the vertices $x$ and $y$).

Comparing the above definition to ``rainbow $k$-connectedness'' defined in the introduction we see that an edge-coloured graph is  rainbow $k$-connected exactly when it is $(k, \infty)$-connected.

We'll need the following lemma which could be seen as a weak analogue of Menger's Theorem. It will allow us to find rainbow paths through prescribed vertices in a highly connected set.
\begin{lemma}\label{MengerLemma}
Let $D$ be a totally coloured digraph and $A$ a $(3kd,d)$-connected subset of $D$.
Let $S$ be a set of colours with $|S|\leq k$ and $a_1, \dots, a_k$ be vertices in $A$ such that no $a_i$ has a colour from $S$ and $a_1, \dots, a_k$ all have different colours.

Then there is a rainbow path $P$ from $a_1$ to $a_k$ of length at most $kd$ which passes through each of $a_1, \dots, a_k$ and avoids $S$
\end{lemma}
\begin{proof}
Using the definition of $(3kd,d)$-connected, there is a rainbow path $P_1$ from $a_1$ to $a_2$ of length $\leq d$ avoiding colours in $S$. Similarly for $i\leq k$, there is a rainbow path $P_i$ from $a_i$ to $a_{i+1}$ of length $\leq d$ internally avoiding colours in $S$ and colours in $P_1, \dots, P_{i-1}$. Joining the paths $P_1, \dots, P_{k-1}$ gives the required path.
\end{proof}

To every coloured directed graph we associate an uncoloured directed graph where two vertices are joined whenever they have a lot of  short paths between them.
\begin{definition}\label{DefinitionDm}
Let $D$ be a totally coloured digraph and $m\in \mathbb{N}$. We denote by $D_m$ the uncoloured directed graph with $V(D_m)=V(D)$ and $xy$ an edge of $D_m$ whenever there are $m$ internally vertex disjoint paths $P_1, \dots, P_m$, each of length $2$ and going from $x$ to $y$ such that $P_1\cup \dots\cup P_m$ is  rainbow.
\end{definition}

It turns out that for properly coloured directed graphs $D$, the uncoloured graph $D_m$ has almost the same minimum degree as $D$. The following lemma will allow us to study short rainbow paths in  coloured graphs by first proving a result about short paths in uncoloured graphs.
\begin{lemma}\label{DCHighDegree}
For all $\epsilon>0$ and $m\in \mathbb{N}$, there is an $N=N(\epsilon,m)=(5m+4)/\epsilon^2$ such that the following holds.
Let $D$ be a properly totally coloured directed graph on at least $N$ vertices with rainbow vertex set. Then we have
$$\delta^+(D_m)\geq \delta^+(D) -\epsilon|D|.$$
\end{lemma}
\begin{proof}
Let $v$ be an arbitrary vertex in $D_m$.
It is sufficient to show that $|N^+_{D_m}(v)|\geq |\delta^+(D)| -\epsilon|D|.$

For $w\in V(D)$, let $r_v(w)= \# \text{rainbow paths of length 2 from } v \text{ to } w.$
Let $W=\{w: r_v(w)\geq 5m\}$.

We show that $W$ is contained in $N^+_{D_m}(v)$.
\begin{claim}\label{WinDC}
If $w \in W$, then we have $vw\in E(D_m)$.
\end{claim}
\begin{proof}
From the definition of $W$, we have $5m$ distinct  rainbow paths $P_1, \dots, P_{5m}$ from $v$ to $w$ of length $2$.
Consider an auxiliary graph $G$ with $V(G)=\{P_1, \dots, P_{5m}\}$ and $P_i P_j\in E(G)$ whenever $P_i\cup P_j$ is  rainbow.

We claim that $\delta(G)\geq 5m-4$. Indeed if for $i\neq j$ we have $P_i=vxw$ and $P_j=vyw$, then, using the fact that the colouring on $D$ is proper and the vertex set is rainbow, it is easy to see that  the only way $P_i\cup P_j$ could not be  rainbow is if one of the following holds:
\begin{align*}
c(vx)=c(yw) \hspace{1cm} 
&c(vx)=c(y) \\
c(vy)=c(xw) \hspace{1cm}
&c(vy)=c(x).
\end{align*}
 Thus if $P_i=vxw$ had five non-neighbours $vy_1w, \dots, vy_5w$ in $G$, then by the Pigeonhole Principle for two distinct $j$ and $k$ we would have one of $c(y_jw)=c(y_kw)$, $c(y_j)=c(y_k)$, or $c(vy_j)=c(vy_k)$. But none of these can occur for distinct paths $vy_jw$ and $vy_kw$ since the colouring on $D$ is proper and the vertex set is rainbow. Therefore $\delta(G)\geq 5m-4$ holds.

Now by Tur\'an's Theorem, $G$ has a clique of size at least $|V(G)|/5=m$. The union of the paths in this clique is  rainbow, showing that $vw\in E(D_m)$.
\end{proof}

Now we show that $W$ is large.
\begin{claim}\label{WLarge}
$|W|\geq \delta^+(D)-\epsilon|D|$
\end{claim}
\begin{proof}
For any $u\in N^+_D(v)$ we let $$N'(u)= N^+_D(u)\setminus\{x\in N^+_D(u):ux \text{ or } x \text{ has the same colour as } v \text{ or } vu\}.$$
Since $D$  is properly coloured, and all the vertices in $D$ have different colours, we have that $|\{x \in N^+(u):ux \text{ or } x \text{ has the same colour as } v \text{ or } vu\}|\leq 4$. This implies that $|N'(u)|\geq \delta^+(D)-4$.

Notice that for a vertex $x$, we have $x \in N'(u)$ if, and only if, the path $vux$ is  rainbow. Indeed $vu$ has a different colour from $v$ and $u$ since the colouring is proper. Similarly $ux$ has a different colour from $u$ and $x$. Finally $ux$ and $x$ have different colours from $v$ and $vu$ by the definition of $N'(u)$.

 Therefore there are $\sum_{u\in N^+_D(v)} |N'(u)|$  rainbow paths of length $2$ starting at $v$ i.e. we have $\sum_{x\in V(D)} r_v(x)=\sum_{u\in N^+_D(v)} |N'(u)|$. For any $x\in D$, we certainly have $r_v(x)\leq |N^+(v)|$. If $x \not\in W$ then we have $r_v(x)< 5m$. Combining these we obtain
$$(|D|-|W|)5m+|W||N^+_D(v)|\geq \sum_{x\in V(D)} r_v(x)=\sum_{u\in N^+(v)} |N'(u)|\geq |N^+_D(v)|(\delta^+(D)-4).$$
The last inequality follows from $|N'(u)|\geq \delta^+(D)-4$ for all $u \in N^+_D(v)$.
Rearranging we obtain
$$|W|\geq \frac{|N^+_D(v)|(\delta^+(D)-4)-5m|D|}{|N^+_D(v)|-5m}\geq\delta^+(D)-\frac{5m|D|}{|N^+_D(v)|}-4\geq \delta^+(D)-(5m+4)\frac{|D|}{\delta^+(D)}.$$
If $(5m+4)/\delta^+(D)\leq\epsilon$, then this implies the claim. Otherwise we have $\delta^+(D)< (5m+4)/\epsilon$ which, since $|D|\geq N_0=(5m+4)/\epsilon^2$, implies that $\delta^+(D)\leq \epsilon|D|$ which also implies the claim.
\end{proof}

Claim~\ref{WinDC} shows that $W\subseteq N^+_{D_m}(v)$, and so Claim~\ref{WLarge} implies that $|N^+_{D_m}(v)|\geq \delta^+(D)-\epsilon|D|$. Since $v$ was arbitrary, this implies the lemma.
\end{proof}

The following lemma shows that every directed graph with high minimum degree contains a large, highly connected subset.
\begin{lemma}\label{LargeConnectedSetLemma}
For all $\epsilon>0$ and $k \in \mathbb{N}$, there is a $d=d(\epsilon)=40\epsilon^{-2}$ and $N=N(\epsilon,s)=32k\epsilon^{-2}$ such that the following holds.
Let $D$ be a directed graph of order at least $N$. Then there is a $(k,d)$-connected subset $A\subseteq V(D)$ satisfying 
$$|A|\geq \delta^+(D)-\epsilon|D|.$$
\end{lemma}
\begin{proof}
We start with the following claim.
\begin{claim}\label{ClaimLargeConnectedSet}
There is a set $\tilde A\subseteq V(D)$ satisfying the following
\begin{itemize}
\item For all $B\subseteq \tilde A$ with $|B|> \epsilon |D|/4$ there is a vertex $v \in \tilde A\setminus B$ such that $|N^+(v)\cap B|\geq \epsilon^2 |D|/16$.
\item $\delta^+(D[\tilde A])\geq \delta^+(D)-\epsilon |D|/4$.
\end{itemize}
\end{claim}
\begin{proof}
Let $A_0= V(D)$.
We define $A_1, A_2, \dots, A_M$ recursively as follows.
\begin{itemize}
\item If $A_i$ contains a set $B_i$ such that $|B_i|> \epsilon|D|/4$ and for all $v \in A_i\setminus B_i$ we have $|N^+(v)\cap B_i|< \epsilon^2 |D|/16$, then we let $A_{i+1}=A_i\setminus B_i$.
\item Otherwise we stop with $M=i$.
\end{itemize}
We will show that that $\tilde A=A_M$ satisfies the conditions of the claim.
Notice that by the construction of $A_M$, it certainly satisfies the first condition.
Thus we just need to show that $\delta^+(D[A_M])\geq \delta^+(D)-\epsilon |D|/4$.

From the definition of $A_{i+1}$ we have that $\delta^+(D[A_{i+1}])\geq \delta^+(D[A_{i}])-\epsilon^2|D|/16$ which implies $\delta^+(D[A_{M}])\geq \delta^+(D)-M\epsilon^2|D|/16$. Therefore it is sufficient to show that we stop with $M\leq 4\epsilon^{-1}$.
This follows from the fact that the sets $B_0, \dots, B_{M-1}$ are all disjoint subsets of $V(D)$ with $|B_i|> \epsilon|D|/4$.
\end{proof}

Let $\tilde A$ be the set given by the above claim.
Let $A=\{v\in \tilde A: |N^-(v)\cap \tilde A|\geq \frac{\epsilon}{2} |D|\}$.
We claim that $A$ satisfies the conditions of the lemma.

To show that $|A|\geq \delta^+(D)-\epsilon|D|$,  notice that we have 
$$\frac{\epsilon}{2}|D|(|\tilde A|-|A|)+|A||\tilde A|\geq \sum_{v\in \tilde A} |N^-(v)\cap \tilde A|= \sum_{v\in \tilde A} |N^+(v)\cap \tilde A|\geq |\tilde A|(\delta^+(D)-\epsilon|D|/4).$$
The first inequality come from bounding $|N^-(v)\cap \tilde A|$ by $\frac{\epsilon}2|D|$ for $v\not\in A$ and by $|\tilde A|$ for $v\in A$.
The second inequality follows from the second property of $\tilde A$ in Claim~\ref{ClaimLargeConnectedSet}.
Rearranging we obtain 
$$|A|\geq \frac{|\tilde A|}{|\tilde A|-\epsilon|D|/2}(\delta^+(D)-3\epsilon|D|/4)\geq \delta^+(D)-\epsilon|D|.$$

Now, we show that $A$ is $(k,d)$-connected in $D$. As in Definition~\ref{DefinitionUncolouredkdCon}, let $S$ be a subset of $V(D)$ with $|S|\leq k-1$ and let $x,y$ be two vertices in $A\setminus S$.
We will find a path of length $\leq d$ from $x$ to $y$ in $\tilde A\setminus S$.
Notice that since $|D|\geq 32k\epsilon^{-2}$, we have $|S|\leq \epsilon^{2} |D|/32$. 

Let $N^t(x)=\{u\in \tilde A\setminus S: d_{D[\tilde A\setminus S]}(x,u)\leq t\}$.
We claim that for all $x\in \tilde A$ and $t\geq 0$ we have
$$|N^{t+1}(x)|\geq \min(|\tilde A|-\epsilon|D|/4, |N^t(x)|+\epsilon^2|D|/32).$$ 
For $t=0$ this holds since we have $|N^1|=|\tilde A|\geq \epsilon |D|/4$.
Indeed if $|N^{t}(x)|< |\tilde A|-\epsilon|D|/4$ holds for some $t$ and $x$, then letting $B=\tilde A\setminus  N^{t}(x)$ we can apply the first property of $\tilde A$ from Claim~\ref{ClaimLargeConnectedSet} in order to find a vertex $u\in N^t(x)$ such that $|N^+(u)\cap (\tilde A\setminus N^t(x))|\geq \epsilon^2 |D|/16$. Using $|S|\leq \epsilon^{2} |D|/32$ we get $|(N^+(u)\setminus S)\cap (\tilde A\setminus N^t(x))|\geq |N^+(u)\cap (\tilde A\setminus N^t(x))|-|S|\geq \epsilon^2 |D|/32$.
Since $(N^+(u)\cap \tilde A\setminus S)\cup N^t(x)\subseteq   N^{t+1}(x)$, we obtain $|N^{t+1}(x)|\geq |N^t(x)|+\epsilon^2|D|/32$.

Thus we obtain that $|N^{t}(x)|\geq \min(|\tilde A|-\epsilon|D|/4, t\epsilon^2|D|/32)$.
Since $(d-1)\epsilon^2/32>1$, we have that $|N^{d-1}(x)|\geq |\tilde A|-\epsilon|D|/4$. 
Recall that from the definition of $A$, we also have also have $|N^-(y)\cap \tilde A|\geq \epsilon |D|/2$. Together these imply that  $N^-(y)\cap N^{d-1}(x)\neq \emptyset$ and hence there is a $x$ -- $y$ path of length $\leq d$ in $\tilde A\setminus S$.
\end{proof}

The following is a generalization of the previous lemma to coloured graphs. This is the main intermediate lemma we need in the proof of Theorem~\ref{MainTheorem}.
\begin{lemma}\label{LargeRainbowConnectedSetLemma}
For all $\epsilon>0$ and $k \in \mathbb{N}$, there is an $d=d(\epsilon)=1280\epsilon^{-2}$ and $N=N(\epsilon,k)=1800k\epsilon^{-4}$ such that the following holds.

Let $D$ be a properly totally coloured directed graph on at least $N$ vertices with rainbow vertex set.
Then there is a $(k,d)$-connected subset $A\subseteq V(D)$ satisfying 
$$|A|\geq \delta^+(D)-\epsilon|D|.$$
\end{lemma}
\begin{proof}
Set $m=9d+3k$, and consider the directed graph $D_m$ as in Definition~\ref{DefinitionDm}.
Using $|D|\geq 1800k\epsilon^{-4}$, we can apply Lemma~\ref{DCHighDegree} with the constant $\epsilon/4$ we have that $\delta^+(D_m)\geq \delta^+(D) -\epsilon|D|/4.$

Apply Lemma~\ref{LargeConnectedSetLemma} to $D_m$ with the constants $\epsilon/4$, and $k$. This gives us a $(k, d/2)$-connected set $A$ in $D_m$ with $|A|\geq \delta^+(D_m)-\epsilon|D_m|/4\geq \delta^+(D)-\epsilon|D|/2$.
We claim that $A$ is $(k, d)$-connected  in $D$. As in Definition~\ref{DefinitionColouredkdCon}, let $S$ be a set of $k$ colours and $x,y \in A$. Let $S_V$ be the vertices of $D$ with colours from $S$. Since all vertices in $D$ have different colours, we have $|S_V|\leq k$.
Since $A$ is $(k, d/2)$-connected in $D_m$, there is a $x$ -- $y$ path $P$ in $(D_m\setminus S_V)+x+y$ of length $\leq d/2$.

Using the property of $D_m$, for each edge $uv\in P$, there are at least $m$ choices for a triple of three distinct colours $(c_1, c_2, c_3)$ and a vertex $y(uv)$ such that there is a path $uy(uv)v$ with $c(uy(uv))=c_1$, $c(y(uv))=c_2$, and $c(y(uv)v)=c_3$. Since $m\geq 9d+3k\geq 6|E(P)|+3|V(P)|+3|S|$, we can choose such a triple for every edge $uv\in P$ such that for two distinct edges in $P$, the triples assigned to them are disjoint, and also distinct from the colours in $S$ and colours of vertices of $P$.

Let the vertex sequence of $P$ be $u,x_1, x_2, \dots, x_{p}, v$. The following sequence of vertices is a rainbow path from $u$ to $v$ of length $2|P|\leq d$ internally avoiding colours in $S$
$$P'=u, y(ux_1), x_1,  y(x_1x_2), x_2, y(x_2x_3), x_3,\dots, x_{p-1},y(x_{p-1}x_p), x_p, y(x_pv), v.$$
To show that $P'$ is a rainbow path we must show that all its vertices and edges have different colours. The vertices all have different colours since the vertices in $D$ all had different colours. The edges of $P'$ all have different colours from each other and the vertices of $P'$ by our choice of the vertices $y(x_ix_{i+1})$ and the triples of colours associated with them.
\end{proof}

We'll need the following simple lemma which says that for any vertex $v$ there is a set of vertices $N^{t_0}$ close to $v$ with few edges going outside $N^{t_0}$.
\begin{lemma}\label{CloseSubgraphsLowExpansion}
Suppose we have $\epsilon>0$ and $D$ a totally coloured directed graph. 
Let $v$ be a vertex in $D$ and for $t\in \mathbb{N}$, let $N^t(v)=\{x:\dRainbow{v,x}\leq t\}$.
There is a $t_0\leq \epsilon^{-1}$ such that we have 
$$|N^{t_{0}+1}(v)|\leq |N^{t_{0}}(v)|+ \epsilon |D|.$$
\end{lemma}
\begin{proof}
Notice that if $|N^{t+1}(v)|> |N^{t}(v)|+ \epsilon |D|$ held for all $t\leq \epsilon^{-1}$, then we would have $|N^{t}(v)|>\epsilon t|D|$ for all $t\leq \epsilon^{-1}$. When $t=\epsilon^{-1}$ this gives $|N^{\epsilon^{-1}}(v)|>|D|$, which is a contradiction.
\end{proof}

A corollary of the above lemma is that for any vertex $v$ in a properly coloured directed graph, there is  a subgraph of $D$ close to  $v$ which has reasonably large minimum out-degree. 
\begin{lemma}\label{CloseHighDegreeSubgraph}
Suppose we have $\epsilon>0$ and $D$ a properly totally coloured directed graph on at least $2\epsilon^{-2}$ vertices, with rainbow vertex set. 
Let $v$ be a vertex in $D$ and let $\delta^+= \min_{x
: \dRainbow{v,x}\leq \epsilon^{-1}} d^+(x)$. Then there is a set $N$ such that $\dRainbow{v,N}\leq \epsilon^{-1}$ and we have
$$\delta^+(D[N])\geq \delta^+ - 2\epsilon |D|.$$
\end{lemma}
\begin{proof}
Apply Lemma~\ref{CloseSubgraphsLowExpansion} to $D$ in order to obtain a number $t_0\leq \epsilon^{-1}$ such that $|N^{t_{0}+1}(v)|\leq |N^{t_{0}}(v)|+ \epsilon |D|$.  We claim that the set $N=N^{t_{0}}(v)$ satisfies the conditions of the lemma. 

Suppose, for the sake of contradiction that there is a vertex $x\in N^{t_0}(v)$ with $|N^+(x)\cap N^{t_0}(v)|< \delta^+ - 2\epsilon |D|$. Since $\delta^+\leq |N^+(x)|$, we have $|N^+(x)\setminus N^{t_0}(v)|> 2\epsilon |D|$. Let $P$ be a length $\leq t_0$ path from $v$ to $x$. Notice that since the colouring on $D$ is proper and all vertices in $D$ have different colours, the path $P+y$ is rainbow for all except at most $2|P|$ of the vertices $y \in N^+(x)$. Therefore we have $|N^+(x)\setminus  N^{t_0+1}(v)|\leq 2|P|\leq 2\epsilon^{-1}$. Combining this with $|D|\geq 2\epsilon^{-2}$, this implies 
\begin{align*}
|N^{t_0+1}(v)|&\geq |N^{t_0}(v)|+|N^+(x)\setminus N^{t_0}(v)|-|N^+(x)\setminus N^{t_0+1}(v)|\\
&>|N^{t_0}(v)|+2\epsilon |D|-2\epsilon^{-1}\\
&\geq |N^{t_0}(v)|+\epsilon|D|.
\end{align*}
This contradicts the choice of $t_0$ in Lemma~\ref{CloseSubgraphsLowExpansion}.
\end{proof}

\section{Proof of Theorem~\ref{MainTheorem}}\label{SectionMainTheorem}
The goal of this section is to prove an approximate version of Conjecture~\ref{ConjectureAharoni} in the case when all the matchings in $G$ are disjoint. The proof will involve considering auxiliary directed graphs to which Lemmas~\ref{LargeRainbowConnectedSetLemma} and~\ref{CloseHighDegreeSubgraph} will be applied.

We begin this section by proving a series of lemmas (Lemmas~\ref{SwitchingLemma} --~\ref{IncrementLemma}) about bipartite graphs consisting of a union of $n_0$ matchings. The set-up for these lemmas will always be the same, and so we state it in the next paragraph to avoid rewriting it in the statement of every lemma.

We will always have bipartite graph called ``$G$'' with bipartition classes $X$ and $Y$ consisting of $n+1$ edge-disjoint matchings $M_1, \dots, M_{n+1}$. These matchings will be referred to as colours, and the colour of an edge $e$ means the matching $e$ belongs to.
There will always be a rainbow matching called $M$ of size $n$ in $G$. 
We set $X_0=X\setminus V(M)$ and $Y_0-Y\setminus V(M)$. The colour missing from $M$ will denoted by $c^*$.

Notice that for any edge $e$, there is a special colour  (the colour $c_e$ of the edge $e$) as well as a special vertex in $X$ (i.e. $e\cap X$) and in $Y$ (i.e. $e\cap Y$). In what follows we will often want to refer to the edge $e$, the colour $c_e$, and the vertices $e\cap X$ and $e\cap Y$ interchangeably. 
To this end we make a number of useful definitions:
\begin{itemize}
\item For an edge $e$, we let $(e)_C$ be the colour of $e$, $(e)_X=e\cap X$, and $(e)_Y=e\cap Y$.
\item For a vertex $x \in X$, we let $(x)_M$ be the edge of $M$ passing through $x$ (if it exists), $(x)_C$ the colour of $(x)_M$, and $(x)_Y$ the vertex $(x)_M\cap Y$. If there is no edge of $M$ passing through $x$, then $(x)_M$, $(x)_C$, and $(x)_Y$ are left undefined.
\item For a vertex $y \in Y$, we let $(y)_M$ be the edge of $M$ passing through $y$ (if it exists), $(y)_C$ the colour of $(y)_M$, and $(y)_X$ the vertex $(y)_M\cap X$. If there is no edge of $M$ passing through $y$, then $(y)_M$, $(y)_C$, and $(y)_X$ are left undefined.
\item For a colour $c$, we let $(c)_M$ be the colour $c$ edge of $M$ (if it exists), $(x)_X=(c)_M\cap X$, and $(c)_Y=(c)_M\cap Y$. For the colour $c^*$, we leave $(c)_M$, $(c)_X$, and $(c)_Y$ undefined.
\end{itemize}
For a set $S$ of colours, edges of $M$, or vertices, we let $(S)_M=\{(s)_M:s\in S\}$, $(S)_X=\{(s)_X:s\in S\}$, $(S)_Y=\{(s)_Y:s\in S\}$, and $(S)_C=\{(s)_C:s\in S\}$. Here $S$ is allowed to contain colours/edges/vertices for which $(*)_M$/$(*)_X$/$(*)_Y$/$(*)_C$ are undefined---in this case $(S)_M$ is just the set of $(s)_M$ for $s\in S$ where $(s)_M$ is defined (and similarly for $(S)_X$/$(S)_Y$/$(S)_C$. 
It is useful to observe that from the above definitions we get identities such as $(((S)_X)_C)_M=S$ for a set $S$ of edges of $M$.

We will now introduce two important and slightly complicated definitions. Both Definition~\ref{DefinitionSwitching} and~\ref{DefinitionFree} will take place in the setting of a bipartite graph $G$ with bipartition $X\cup Y$ consisting of $n+1$ edge-disjoint matchings, and a rainbow matching $M$ of size $n$ missing colour $c^*$.
The first definition is that of a \emph{switching}---informally this should be thought of as a sequence of edges of $G\setminus M$ which might be exchanged with a sequence of edges of $M$ in order to produce a new rainbow matching of size $n$. See Figure~\ref{SwitchingFigure} for an illustration of a switching.

\begin{figure}
  \centering
     \includegraphics[width=0.6\textwidth]{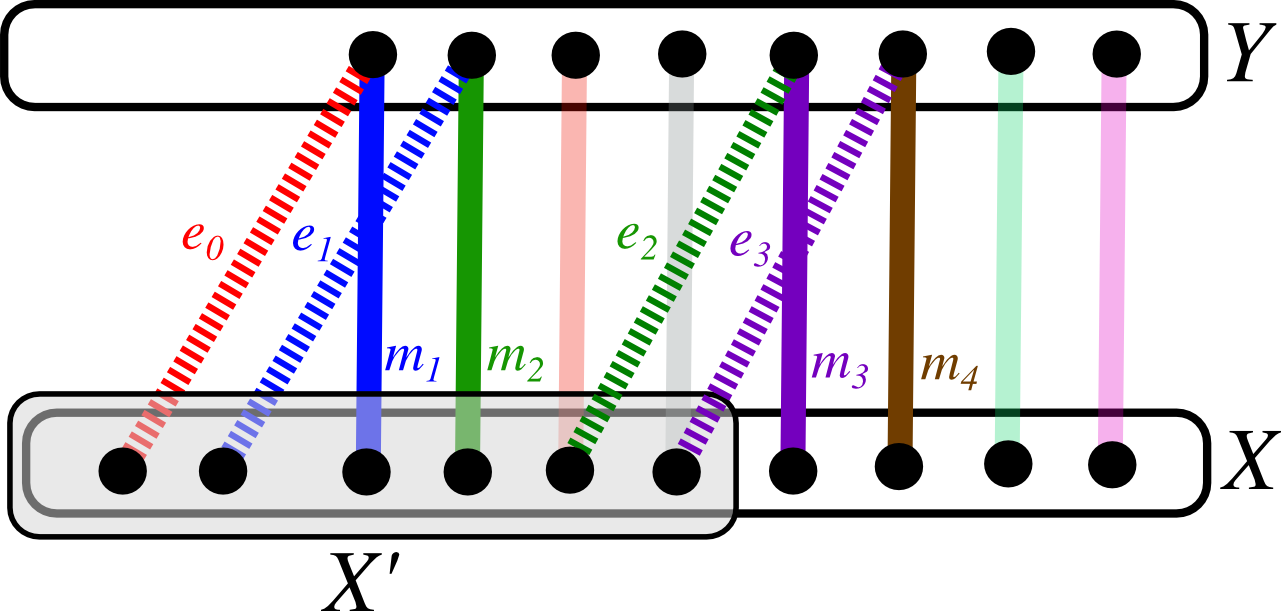}
  \caption{An $X'$-switching of length $4$. The solid lines represent edges of $M$ and the dashed lines represent edges not in $M$. \label{SwitchingFigure}}
\end{figure}
\begin{definition}\label{DefinitionSwitching}
Let $X'\subseteq X$. We call a sequence of edges, $\sigma=(e_0, m_1, e_1, m_2, e_2, \dots, e_{\ell-1}, m_{\ell})$, an $X'$-switching if the following hold.
\begin{enumerate}[(i)]
\item For all $i$, $m_i$ is an edge of $M$ and $e_i$ is not an edge of $M$.
\item For all $i$, $m_i$ and $e_i$ have the same colour, $c_i$.
\item For all $i$, $e_{i-1}\cap m_i=(m_i)_Y$.
\item For all $i\neq j$, we have $e_{i}\cap e_j=e_{i-1}\cap m_{j}=\emptyset$ and also $c_i\neq c_j$.
\item For all $i$, $(e_i)_X\in X'$.
\end{enumerate}
\end{definition}
If $\sigma$ is a switching defined as above, then we say that $\sigma$ is a length $\ell$ switching from $c_0$ to $c_{\ell}$.
Let $e(\sigma)=\{e_0, \dots, e_{\ell-1}\}$ and $m(\sigma)=\{m_1, \dots, m_{\ell}\}$. For a switching $\sigma$ we define $(\sigma)_X=(e(\sigma))_X\cup (m(\sigma))_X$.

The next definition is that of a \emph{free} subset of $X$---informally a subset $X'\subset X$ is free if there are matchings $M'$ which ``look like'' $M$, but avoid $X'$.
\begin{definition}\label{DefinitionFree}
Let $X', T\subseteq X$, $k\in\mathbb{N}$, and $c$ be a colour.
We say that $X'$ is $(k,T,c)$-free if $T\cap X'=\emptyset$, $c \not\in (X'\cup T)_C$, and the following holds:

Let $A$ be any set of $k$ edges in $M\setminus ((T)_M\cup (c)_M)$, $B\subseteq X'$ any set of $k$ vertices such that $(A)_X\cap B=\emptyset$. Then there is a rainbow matching $M'$ of size $n$ satisfying the following:
\begin{itemize}
\item $M'$ agrees with $M$ on  $A$.
\item $M'\cap B=\emptyset$.
\item $M'$ misses the colour $c$.
\end{itemize}
\end{definition}
It is worth noticing that $X_0$ is $(n,\emptyset,c^*)$-free (always taking the matching $M'$ to be $M$ in the definition). Intuitively free sets should be thought of as sets which ``behave like $X_0$'' for the purposes of finding a matching larger than $M$.

The following lemma is crucial---it combines the preceding two definitions together and says that if we have an $X'$-switching $\sigma$ for a free set $X'$, then there is a new rainbow matching of size $n$ which avoids $(m(\sigma))_X$.
\begin{lemma}\label{SwitchingLemma}
Suppose that $X'$ is $(2k,T,c)$-free and $\sigma=(e_0, m_1, e_1, \dots, e_{\ell-1}, m_{\ell})$ is an $X'$-switching from $c$ to $(m_{\ell})_C$ of length $\ell\leq k$. 
 Let $A$ be any set of at most $k$ edges in $M-(c)_M$ and let $B$ be any subset of $X'$ of order at most $k$. Suppose that the following disjointness conditions hold
  \begin{align*}
\hspace{2cm} &(\sigma)_X\cap T=\emptyset &(\sigma)_X\cap (A)_X=\emptyset \hspace{1cm} &(\sigma)_X\cap B=\emptyset \hspace{2cm}\\
\hspace{2cm} & &T\cap(A)_X =\emptyset \hspace{1cm}  &(A)_X\cap B=\emptyset.\hspace{2cm}
 \end{align*}
Then there is a rainbow matching $\tilde M$ of size $n$ in $G$ which misses colour $(m_{\ell})_C$, agrees with $M$ on $A$, and has $\tilde M\cap (m(\sigma))_X=\tilde M\cap  B=\emptyset$.
\end{lemma}
\begin{proof}
We let $A'=m(\sigma)\cup A$ and $B'=(e(\sigma))_X\cup B$.
Notice that we have $|A'|, |B'|\leq 2k$. Also from the definition of ``switching'', we have that for any $i$ and $j$, the edges $e_i$ and $m_j$ never intersect in $X$ which together with $(A)_X\cap (\sigma)_X=\emptyset$, $B\cap (\sigma)_X=\emptyset$, and $(A)_X\cap B=\emptyset$ implies that $(A')_X\cap B'=\emptyset$.
Also, $(\sigma)_X\cap T=\emptyset$ and $(A)_X\cap T=\emptyset$ imply that  $A'\cap (T)_M=\emptyset$, and hence since $\sigma$ is a switching starting at $c$ we have  $A'\subseteq M\setminus ((T)_M\cup (c)_M)$.  Finally, $\sigma$ being an $X'$-switching and $B\subseteq X'$ imply that $B'\subseteq X'$.

Therefore we can invoke the definition $X'$ being $(2k,T,c)$-free in order to obtain a rainbow matching $M'$ of size $n$ avoiding $B'$, agreeing with $M$ on $A'$, and missing colour $c=(e_{0})_C$.
We let $$\tilde M=(M'\setminus m(\sigma))\cup e(\sigma) = M' +e_0- m_1+ e_1- m_2+ e_2 - \dots+ e_{\ell-1}- m_{\ell}.$$

We claim that $\tilde M$ is a matching which satisfies all the conditions of the lemma.

Recall that $B'\supseteq(e(\sigma))_X$, $A'\supseteq m(\sigma)$, and $(A')_X\cap B'=\emptyset$. Since $M'$ agreed with $M$ on $A'$ and was disjoint from $B'$, 
we get $m(\sigma)\subseteq M'$ and $e(\sigma)\cap (M'\setminus m(\sigma))=\emptyset$.
This implies that $\tilde M$ is a set of $n$ edges and also that $(\tilde M)_X=\big((M')_X\setminus (m(\sigma))_X\big)\cup (e(\sigma))_X$ is a set of $n$ vertices. Finally notice that since $(e_i)_Y=(m_{i+1})_Y$ we have $(\tilde M)_Y=(M')_Y$. Thus $\tilde M$ is a set of $n$ edges with $n$ vertices in each of $X$ and $Y$ i.e. a matching. The matching $\tilde M$ is clearly rainbow, missing the colour $(m_{\ell})_C$ since $m_i$ and $e_i$ always have the same colour.

To see that $\tilde M$ agrees with $M$ on edges in $A$, notice that $M'$ agreed with $M$ on these edges since we had $A\subseteq A'$. Since $(\sigma)_X\cap (A)_X=\emptyset$ implies that $\sigma$ contains no edges of $A$,  we obtain that $\tilde M$ agrees with $M$ on $A$.

%\textbf{REPHRASE}
%It remains to show that  $\tilde M\cap (m(\sigma))_X=\tilde M\cap B=\emptyset$.
To see that $\tilde M\cap (m(\sigma))_X=\emptyset$, recall that $(\tilde M)_X=\big((M')_X\setminus (m(\sigma))_X\big)\cup (e(\sigma))_X$ and $(m(\sigma))_X\cap (e(\sigma))_X=\emptyset$.
Finally, $\tilde M\cap B=\emptyset$ follows from $M'\cap B=\emptyset$, $(\tilde M)_X=\big((M')_X\setminus (m(\sigma))_X\big)\cup (e(\sigma))_X$, and $B\cap (\sigma)_X=\emptyset$.
\end{proof}

We study $X'$-switchings by looking at an auxiliary directed graph.
For any $X'\subseteq X$, we will define a directed, totally labelled graph $D_{X'}$. We call $D_{X'}$ a ``labelled'' graph rather than a ``coloured'' graph just to avoid confusion with the coloured graph $G$. Of course the concepts of ``coloured'' and ``labelled'' graphs are equivalent, and we will freely apply results from Section~\ref{SectionConnectedness} to labelled graphs. 
The vertices and edges of $D_{X'}$ will be labelled by elements of the set $X\cup \{*\}$.  
\begin{definition}
Let $X'$ be a subset of $X$. The directed graph $D_{X'}$ is defined as follows:
\begin{itemize}
\item The vertex set of $D_{X'}$ is the set of colours of edges in $G$. For any colour $v\in V(D_{X'})$ present in $M$, $v$ is labelled by ``$(v)_X$''. The colour $c^*$ is labelled by ``$*$''.
\item For two colours $u$ and $v\in V(D_{X'})$, there is a directed edge from $u$ to $v$ in $D_{X'}$ whenever there is an $x\in X'$ such that there is a colour $u$ edge from $x$ to the the vertex $(v)_Y$ in $G$. In this case $uv$ is labelled by ``$x$''.
\end{itemize}
\end{definition}
Notice that in the second part of this definition the labelling is well-defined since there cannot be colour $u$ edges from two distinct vertices $x$ and $x'$ to $(v)_Y$ (since the colour $u$ edges form a matching in $G$).

Recall that a total labelling is proper if outgoing edges at a vertex always have different labels, ingoing edges at a vertex always have different labels, adjacent vertices have different labels, and an edge always has different labels from its endpoints.
Using the fact that the matchings in $G$ are disjoint we can show that $D_{X'}$ is always properly labelled.
\begin{lemma}\label{ProperColouring}
For any $X'\subseteq X$ the total labelling on $D_{X'}$ is always proper. In addition  $D_{X'}$ has rainbow vertex set.
\end{lemma}
\begin{proof}
Suppose that $uv$ and $u'v'$ are two distinct edges of $D_{X'}$ with the same label $x\in X'$.
By definition of $D_{X'}$ they correspond to two  edges $x(v)_Y$ and $x(v')_Y$ of $G$ having colours $u$ and $u'$ respectively. This implies that $u$ and $u'$ are different since otherwise we would have two edges of the same colour leaving $x$ in $G$ (which cannot happen since colour classes in $G$ are matchings). We also get that $v$ and $v'$ are distinct since otherwise we would have edges of colours both $u$ and $u'$ between $x$ and $(v)_Y$ in $G$ (contradicting the matchings forming $G$ being disjoint).

Let $uv$ be an edge of $D_{X'}$ labelled by $x$ and $x(v)_Y$ the corresponding colour $u$ edge of $G$. Then  $u$ cannot be labelled by ``$x$'' (since that would imply that the colour $u$ edge at $x$ would end at $(u)_Y$ rather than $(v)_Y$), and $v$ cannot be labelled by ``$x$'' (since then  there would be edges from $x$ to $(v)_Y$ in $G$ of both colours $u$ and $v$).

The fact that $D_X'$ has rainbow vertex set holds because $M$ being a matching implies that $(c)_X$ is distinct for any colour $c$.
\end{proof}

Recall that a path in a totally labelled graph is defined to be rainbow whenever all its vertices and edges have different colours.
The reason we defined the directed graph $D_{X'}$ is that rainbow paths in $D_{X'}$ correspond exactly to $X'$-switchings in $G$.  Let $P=v_0, \dots, v_{\ell}$ be a path in $D_{X'}$ for some $X'$. For each $i=0, \dots, \ell-1$ let $e_i$ be the colour $v_i$ edge of $G$ corresponding to the edge $v_{i} v_{i+1}$ in $D_{X'}$. 
We define $\sigma_P$ to be the sequence of edges $(e_0,$ $(v_{1})_M, e_1,$ $(v_{2})_M,$ $e_2, \dots, (v_{\ell-1})_M, e_{\ell-1},$ $(v_{\ell})_M)$. Notice that $(e(\sigma_P))_X$ is the set of labels of edges in $P$, and $(m(\sigma_P))_X$ is the set of labels of vertices in $P-v_0$.

The following lemma shows that if $P$ is rainbow then $\sigma_P$ is a switching.
\begin{lemma}\label{PathSwitching}
Let $P=v_0, \dots, v_{\ell}$ be a rainbow path in $D_{X'}$ for some $X'\subseteq X$.
Then $\sigma_P$ is an $X'$-switching from $v_0$ to $v_\ell$ of length $\ell$.
\end{lemma}
\begin{proof}
As in the definition of  $\sigma_P$, let $e_i$ be the colour $v_i$ edge of $G$ corresponding to the edge $v_{i} v_{i+1}$ in $D_{X'}$.

We need to check all the parts of the definition of ``$X'$-switching''. 
For part (i), notice that $(v_{1})_M, \dots, (v_{\ell})_M$ are edges of $M$ by definition of $(.)_M$, whereas $e_i$ cannot be the colour $v_i$ matching edge $(v_i)_M$ since $(e_i)_Y=(v_{i+1})_M\cap Y$ which is distinct from $(v_i)_M\cap Y$.
Parts (ii), (iii), and (v) follow immediately from the definition of $e_i$ and the graph $D_{X'}$. 

Part (iv) follows from the fact that $P$ is a rainbow path. Indeed to see that for $i\neq j$ we have $e_i\cap e_j=\emptyset$, notice that $e_i\cap e_j\cap X=\emptyset$ since $v_i v_{i+1}$ and $v_j v_{j+1}$ have different labels in $D_{X'}$, and that $e_i\cap e_j\cap Y=\emptyset$ since $(e_i)_Y\in (v_{i+1})_M$, $(e_j)_Y\in (v_{j+1})_M$, and $(v_{i+1})_M\cap(v_{j+1})_M=\emptyset$. Similarly for $i\neq j$, $e_{i-1}\cap (v_j)_M\cap X=\emptyset$ since $v_{i-1}v_{i}$ and $v_j$ have different labels in $D_{X'}$, and $e_{i-1}\cap (v_j)_M\cap Y=\emptyset$ since $(e_{i-1})_Y\in (v_{i})_M$ and $(v_{i})_M\neq (v_{j})_M$. Finally, $c_i\neq c_j$ since $v_0, \dots, v_{\ell}$ are distinct.
\end{proof}
Although it will not be used in our proof, it is worth noticing that the converse of Lemma~\ref{PathSwitching} holds i.e. to every $X'$-switching $\sigma$ there corresponds a unique rainbow path $P$ in $D_{X'}$ such that $\sigma=\sigma_P$.

So far all our lemmas were true regardless whether the rainbow matching $M$ was maximum or not. Subsequent lemmas will assume that $M$ is maximum. 
The following lemma shows that for a free set $X'$, vertices in $D_{X'}$ have large out-degree.
\begin{lemma}\label{ShortPathHighDegree}
Suppose there is no rainbow matching  in $G$ of size $n+1$.
Let $X'$, $T$, $k$ and $c$ be such that $X'$ is $(2k,T,c)$-free.
Let $D=D_{X'}\setminus (T)_C$, $v$ a vertex of $D$, and  $P$ a rainbow path in $D$ from $c$ to $v$ of length at most $k$. 
Then we have 
$$|N^+_D(v)|\geq (1+\epsilon_0)n+|X'|-|X|-2|P|-|T|.$$
\end{lemma}
\begin{proof}
Notice that since $P$ is contained in $D_{X'}\setminus (T)_C$ and since $X'$ being $(k,T,c)$-free  implies $X'\cap T=\emptyset$, we can conclude that $(\sigma_P)_X\cap T=\emptyset$. 

Therefore, Lemma~\ref{SwitchingLemma} applied with $A=\emptyset$ implies that for any  $B\subseteq X'$ with $|B|\leq k$ and $B\cap (\sigma_P)_{X}=\emptyset$, there is a rainbow matching $M'$ of size $n$ which is disjoint from $B$ and misses colour $v$. Since there are no rainbow matchings of size $n+1$ in $G$ this means that there are no colour $v$ edges from $X'\setminus (\sigma_P)_{X}$ to $Y_0$ (indeed if such an edge $xy$ existed, then we could apply Lemma~\ref{SwitchingLemma} with $B=\{x\}$ in order to obtain a rainbow matching $M'$ missing colour $v$ and vertex $x$ which can be extended to a rainbow $n+1$ matching by adding the edge $xy$).

We claim that there are at least $(1+\epsilon_0)n+|X'|-|X|-2|P|$ colour $v$ edges from $X'\setminus (\sigma_P)_X$. Indeed out of the $(1+\epsilon_0)n$ colour $v$ edges in $G$ at most $|X|-|X'|$ of them can avoid $X'$, and at most $2|P|$ of them can pass through $(\sigma_P)_X$, leaving at least $(1+\epsilon_0)n-(|X|-|X'|)-2|P|$ colour $v$ edges to pass through  $X'\setminus (\sigma_P)_X$. Since none of these edges can touch $Y_0$, each of them must give rise to an out-neighbour of $v$ in $D_{X'}$. This shows that $|N^+_{D_{X'}}(v)|\geq (1+\epsilon_0)n+|X'|-|X|-2|P|$ which implies the result.
\end{proof}

The following lemma is the essence of the proof of Theorem~\ref{MainTheorem}. It roughly says that given a free set $X_1$ containing $X_0$, there another free set $X_2$ containing $X_0$ such that $X_2$ is much bigger than $X_1$, but has worse parameter $k$. The proof of this lemma combines everything in this section with Lemmas~\ref{MengerLemma},~\ref{LargeRainbowConnectedSetLemma} and~\ref{CloseHighDegreeSubgraph} from Section~\ref{SectionConnectedness}.

\begin{lemma}\label{IncrementLemma}
Let $k_1$ be an integer such that $n\geq 10^{20} \epsilon_0^{-8}k_1$ and $k_1\geq 20\epsilon_0^{-1}$. Set $k_2=10^{-6}\epsilon_0^{2}k_1$.
Suppose there is no rainbow matching  in $G$ of size $n+1$.
\begin{itemize}
\item Suppose that we have $X_1, T_1 \subseteq X$ and a colour $c_1$ such that $X_1$ is $(k_1, T_1, c_1)$-free and we also have $X_0\subseteq X_1\cup T_1$ and $|T_1|\leq k_1-30\epsilon_0^{-1}$.
\item Then there are $X_2, T_2\subseteq X$ and a colour $c_2$ such that $X_2$ is $(k_2, T_2, c_2)$-free and we also have $X_0\subseteq X_2\cup T_2$, $|T_2|\leq |T_1|+ 30\epsilon_0^{-1}$ and
$$|X_2|> |X_1|+\frac{\epsilon_0}2n.$$
\end{itemize}
\end{lemma}
\begin{proof}
Set $d=10^5\epsilon_0^{-2}$.
Let $D=D_{X_1}\setminus (T_1)_C$. Recall that Lemma~\ref{ProperColouring} implies that $D$ is properly labelled with rainbow vertex set.

Lemma~\ref{ShortPathHighDegree}, together with  $n\geq 10^{20} \epsilon_0^{-8}k_1$, $k_1\geq 20\epsilon_0^{-1}$, and $|T_1|\leq k_1$ imply that all vertices in $D$ within rainbow distance $(10\epsilon_0)^{-1}$ of $c_1$  satisfy $d^+(v)\geq (1+\epsilon_0)n+|X_1|-|X|-30\epsilon_0^{-1}\geq(1+0.9\epsilon_0)n+|X_1|-|X|$. 

Lemma~\ref{CloseHighDegreeSubgraph} applied with $\epsilon=0.1\epsilon_0$ implies that there is a subgraph $D'$  in $D$ satisfying $\delta^+(D')\geq (1+0.7\epsilon_0)n+|X_1|-|X|$ and $\dRainbow{c_1, v}\leq 10\epsilon_0^{-1}$ for all $v\in D'$. 
Therefore, using $n\geq 10^{20} \epsilon_0^{-8}k_1$,  we can apply Lemma~\ref{LargeRainbowConnectedSetLemma} to $D'$ with $\epsilon=0.1\epsilon_0$ and $k=9k_2d$ in order to find a  set $W$ with $|W|\geq(1+0.6\epsilon_0)n+|X_1|-|X|$ which is $(9k_2 d, d)$-connected in $D'$. 

Since $W\subseteq D'$, there is a path, $Q$, of length $\leq 10\epsilon_0^{-1}$ from $c_1$ to some $q\in W$. Let $c_2$ be any vertex in $W$ with $(c_2)_X\not \in (\sigma_Q)_X$. Let $T_2=T_1\cup (\sigma_Q)_X\cup (c_1)_X$. 
Let $X_2=((W)_X\cup X_0)\setminus (T_2\cup (c_2)_X)$.
We claim that $X_2$, $T_2$, and $c_2$ satisfy the conclusion of the lemma.

First we show that $X_2$  is $(k_2, T_2, c_2)$-free. 
The facts that  $T_2\cap X_2=\emptyset$ and $c_2\not\in X_2\cup T_2$ follow from the construction of $X_2$, $T_2$, and $c_2$.
Let $A$ be any set of $k_2$ edges of $M\setminus ((T_2)_M\cup (c_2)_M)$, and $B\subseteq X_2$ any set of $k_2$ vertices such that $(A)_X\cap B=\emptyset$. 
Let $B_{X_0}=B\cap X_0$ and $B_{W}=B\cap (W)_X=B\setminus B_{X_0}$.
By Lemma~\ref{MengerLemma}, applied with $k=3k_2$, $d=d$, $A=W$, $\{q, a_1, \dots, a_k, c_2\}=(B_W)_C$, and $S=(A)_X\cup (\sigma_Q)_X\cup B_{X_0}$, there is a rainbow path $P$ in $D'$ of length $\leq 3k_2 d$ from $q$ to $c_2$ which is disjoint from $V(Q-q)$ and $(A)_C$, passes through every colour of $(B_W)_C$, and whose edges and vertices don't have labels in $(A)_X\cup (\sigma_Q)_X\cup B_{X_0}\setminus (q)_X$. Notice that this means that $Q+P$ is a rainbow path from $c_1$ to $c_2$.

We apply Lemma~\ref{SwitchingLemma} with $X'=X_1$, $T=T_1$, $c=c_1$, $\sigma=\sigma_{Q+P}$, $A=A$, $B=B_{X_0}$. For this application notice that $\sigma_{Q+P}$ is a $X_1$-switching of length $\leq k_1/2$, which holds because of Lemma~\ref{PathSwitching} and because $2|Q|+2|P|\leq 20\epsilon_0^{-1}+2k_2d\leq k_1/2$.
We also need to check the various disjointness conditions---$(A)_X\cap T_1=(A)_X\cap (\sigma_{Q+P})_X=(A)_X\cap B_{X_0}=\emptyset$ (which hold because $(A)_X$ was disjoint from $T_2$, $P$, and $B$), $(\sigma_{Q+P})_X\cap T_1=\emptyset$ (which holds since vertices and edges in $D$ have no labels from $T_1$),  and $(\sigma_{Q+P})_X\cap B_{X_0}=\emptyset$ (which holds since $B$ was disjoint from $T_2$ and $P$ had no labels from $B_{X_0}$).
Therefore Lemma~\ref{SwitchingLemma} produces a rainbow matching $M'$ of size $n$ which agrees with $M$ on $A$, avoids $(m(\sigma_{Q+P}))_X\cup B_{X_0}$, and misses colour $c_2$. Since $P$ passes through every colour in $(B_W)_C$, we have $B_W\subseteq (m(\sigma_{Q+P}))_X$ and so $M'$ avoids all of $B$.
Since $A$ and $B$ were arbitrary, we have shown that $X_2$ is $(k_2,T_2,c_2)$-free.

The identity $X_0\subseteq X_2\cup T_2$ holds because $X_0\subseteq  X_1\cup T_1\subseteq X_2\cup T_2$.
Notice that $|T_2|\leq |T_1|+ 30\epsilon_0^{-1}$ follows from $|Q|\leq 10\epsilon^{-1}$.

Finally, $|X_2|> |X_1|+\epsilon_0n/2$  holds because since $(W)_X$ was disjoint from $X_0$ we have 
$$|X_2|\geq |X_0|+|W|\geq |X_0|+(1+0.6\epsilon_0)n+|X_1|-|X|=|X_1|+ 0.6\epsilon_0n.$$
%The last inequality follows from $|X|=|X_0|+n$.
\end{proof}

We are finally ready to prove Theorem~\ref{MainTheorem}. The proof consists of starting with $X_0$ and applying Lemma~\ref{IncrementLemma} repeatedly, at each step finding a free set $X_i$ which is $\epsilon n/2$ bigger than $X_{i-1}$. This clearly cannot be performed more than $2\epsilon_0 ^{-1}$ times (since otherwise it would contradict $|X_i|\leq |X|=|X_0|+n$), and hence the ``there is no rainbow matching  in $G$ of size $n+1$'' clause of Lemma~\ref{IncrementLemma} could not be true.
\begin{proof}[Proof of Theorem~\ref{MainTheorem}] 
Let $G$ be a bipartite graph which is the union of $n_0\ge N_0$ disjoint matchings each of size at least $(1+\epsilon_0)n_0$. Let $M$ be the largest rainbow matching in $G$ and $c^*$ the colour of any matching not used in $M$.  Let $n$ be the number of edges of $M$. Since $M$ is maximum, Lemma~\ref{GreedyLemma} tells us that $n\geq N_0/2$. Let $X_0=X\setminus M$ and $Y_0=Y\setminus M$. Suppose for the sake of contradiction that $n<n_0$.

Let $T_0=\emptyset$, $k_0=(10^{-6}\epsilon^{-2})^{2\epsilon^{-1}}$, and $c_0=c^*$. Notice that since $X_0$ is $(n, T_0, c_0)$-free and $n\geq N_0/2\geq k_0$ we get that $X_0$ is $(k_0, T_0, c_0)$-free.
For $i=1, \dots, 2\epsilon^{-1}$, we set $k_{i}=10^{-6}\epsilon_0^{2}k_{i-1}$. %Notice that we always have $n\geq 2015 k_i\epsilon_0^{-4}$ and $k_i\geq 10\epsilon_0^{-4}$.

For $i=0, \dots, 2\epsilon^{-1}$ we repeatedly apply Lemma~\ref{IncrementLemma} to $X_i$, $k_i$, $T_i$, $c_i$ in order to obtain sets  $X_{i+1}$, $T_{i+1}\subseteq X$ and a colour $c_{i+1}$ such that $X_{i+1}$ is $(k_{i+1}, T_{i+1}, c_{i+1})$-free, $X_0\subseteq X_{i+1}\cup T_{i+1}$, $|T_{i+1}|\leq |T_i|+30\epsilon_0^{-1}$, and $|X_{i+1}|>|X_i|+\epsilon_0n/2$. To see that we can repeatedly apply   Lemma~\ref{IncrementLemma} this way we only need to observe that there are no rainbow $n+1$ matchings in $G$, and that for $i\leq 2\epsilon^{-1}$ we always have $n\geq 10^{20} \epsilon_0^{-8}k_i$,  $k_i\geq 10\epsilon_0^{-1}$, and $|T_i|\leq 30\epsilon^{-1}i\leq k_i-30\epsilon^{-1}$.

But now we obtain that $|X_{2\epsilon^{-1}}|> |X_0|+n=|X|$ which is a contradiction since $X_i$ is a subset of $X$.
\end{proof}

\section{Golden Ratio Theorem}\label{SectionGoldenRatio}
In this section we prove Theorem~\ref{GoldenRatioTheorem}. The proof uses Theorem~\ref{Woolbright} as well as Lemma~\ref{CloseSubgraphsLowExpansion}.
\begin{proof}[Proof of Theorem~\ref{GoldenRatioTheorem}.]
The proof is by induction on $n$. The case ``$n=1$'' is trivial since here $G$ is simply a matching.
Suppose that the theorem holds for all $G$ which are unions of $<n$ matchings.
Let $G$ be a graph which is the union of $n$ matchings each of size $\phi n+ 20n/ \log n$. Suppose that $G$ has no rainbow matching of size $n$.
Let $M$ be a maximum rainbow matching in $G$. By induction we can suppose that $|M|= n-1$. Let $c^*$ be the missing colour in $M$.

Let $X_0=X\setminus V(M)$ and $Y_0=Y\setminus V(M)$.  Notice that for any colour $c$ there are at least $(\phi-1) n+20n/\log n$ colour $c$ edges from $X_0$ to $Y$ and at least at least $(\phi-1) n+20n/\log n$ colour $c$ edges from $Y_0$ to $X$. If $n< 10^6$, then this would give more than $n$ colour $c^*$ edges from $X_0$ to $Y$, one of which could be added to $M$ to produce a larger matching. Therefore, we have that $n\geq 10^6$.

We define an edge-labelled directed graph $D$ whose vertices are the colours in $G$, and whose edges are labelled by vertices from $X_0\cup Y_0$. We set $cd$ an edge in $D$ with label $v\in X_0\cup Y_0$ whenever there is a colour $c$ edge from $v$ to the colour $d$ edge of $M$. Notice that $D$ is out-proper---indeed if edges $ux$ and $uy\in E(D)$ had the same label $v\in X_0\cup Y_0$, then they would correspond to two colour $u$ edges touching $v$ in $G$ (which cannot happen since the colour classes of $G$ are matchings).

Recall that $\dRainbow{x,y}$ denotes the length of the shortest rainbow $x$ to $y$ path in $D$.

We'll need the following two claims.
\begin{claim}\label{GoldRatioFewXYedges}
For every $c\in V(D)$, there are at most $\dRainbow{c^*,c}$ colour $c$ edges between $X_0$ and $Y_0$.
\end{claim}
\begin{proof}
Let $P=c^*p_1 \dots, p_k c$ be a rainbow path of length $\dRainbow{c^*,c}$ from $c^*$ to $c$ in $D$.  For each $i$, let $m_i$ be the colour $p_i$ edge of $M$, and let $e_i$ be the colour $p_i$ edge from the label of $p_ip_{i+1}$ to $m_{i+1}$. Similarly, let $e_{c^*}$ be the colour $c^*$ edge from the label of $c^*p_1$ to $m_{1}$, and let $m_{c}$ be the colour $c$ edge of $M$.
If there are more than $\dRainbow{c^*,c}$ colour $c$ edges between $X_0$ and $Y_0$, then there has to be at least one such edge, $e_{c}$, which is disjoint from $e_{c^*}, e_1, \dots, e_{k}$. Let
$$M'=M+e_{c^*}-m_1+e_1-m_2+e_2\dots -m_{k-1}+e_{k-1}-m_{c}+e_{c}.$$
The graph $M'$ is clearly a rainbow graph with $n$ edges. We claim that it is a matching. Distinct edges $e_i$ and $e_j$ satisfy $e_i\cap e_j=\emptyset$ since $P$ is a rainbow path. The edge $e_i$ intersects $V(M)$ only in one of the vertices of $m_i$, which are not present in $M'$. This means that $M'$ is a rainbow matching of size $n$  contradicting our assumption that $M$ was maximum.
\end{proof}

\begin{claim}\label{CloseSubgraphGoldenRatio}
There is a set $A\subseteq V(D)$ containing $c^*$ such that for all $v\in A$  we have $|N^+(a)\setminus A|\leq  n/\log n$ and  $\dRainbow{c,v} \leq \log n$.
\end{claim}
\begin{proof}
This follows by applying Lemma~\ref{CloseSubgraphsLowExpansion} to $D$ with $\epsilon=(\log n)^{-1}$.
\end{proof}

Let $A$ be the set of colours given by the above claim. Let $M'$ be the submatching of $M$ consisting of the edges with colours not in $A$. Since $c^* \in A$, we have $|M'|+|A|=n$.

Let $A_X$ be the subset of $X$ spanned by edges of $M$ with colours from $A$, and $A_Y$ be the subset of $Y$ spanned by edges of $M$ with colours from $A$.
 Claim~\ref{GoldRatioFewXYedges} shows that for any $a\in A$ there are at most $\log n$ colour $a$ edges between $X_0$ and $Y_0$.  Therefore there are at least $(\phi-1) n+20n/\log n-\log n$ colour $a$ edges from $X_0$ to $Y\cap M$. Using the property of $A$ from Claim~\ref{CloseSubgraphGoldenRatio} we obtain that there are at least $(\phi-1) n+19n/\log n-\log n$ colour $a$ edges from $X_0$ to $A_Y$. Similarly, for any $a\in A$ we obtain at least $(\phi-1) n+19n/\log n-\log n$ colour $a$ edges from $Y_0$ to $A_X$.

By applying Theorem~\ref{Woolbright} to the subgraph of $G$ consisting of the colour $A$ edges between $X_0$ and $A_Y$ we can find a subset $A_0\subseteq A$  and a rainbow matching $M_0$ between $X_0$ and $A_Y$ using exactly the colours in $A_0$ such that we have
\begin{align*}
|A_0|&\geq  (\phi-1)n+19n/\log n-\log n -\sqrt{(\phi-1)n+19n/\log n-\log n}\\
&\geq (\phi-1)n-6\sqrt{n}
\end{align*}

Let $A_1=A\setminus A_0$. We have $|A_1|\leq n-|A_0|\leq (2-\phi)n+ 6\sqrt{n}$. 
Recall that for each $a\in A_1$ there is a colour  $a$ matching between $Y_0$ and $A_X$ of size at least $(\phi-1) n+19n/\log n-\log n$. 
Notice that the following holds
\begin{align*}
(\phi-1) n+\frac{19n}{\log n} -\log n&\geq  \phi ((2-\phi)n+ 6\sqrt{n})+\frac{20((2-\phi)n+ 6\sqrt{n})}{ \log((2-\phi)n+ 6\sqrt{n})}\\
&\geq \phi|A_1|+\frac{20|A_1|}{\log |A_1|}.
\end{align*}
The first inequality follows from $\phi^2-\phi-1=0$ as well as some simple bounds on $\sqrt n$ and $\log n$ for $n\geq 10^6$. The second inequality holds since $x/\log x$ is increasing.

By induction there is a rainbow matching $M_1$ between $Y_0$ and $A_X$ using exactly the colours in $A_1$.
Now $M'\cup M_0\cup M_1$ is a rainbow matching in $G$ of size $n$.
\end{proof}
\section{Concluding remarks}\label{SectionConclusion}
Here we make some concluding remarks about the techniques used in this paper.
\subsection*{Analogues of Menger's Theorem for rainbow $k$-connectedness}
 One would like to have a version of Menger's Theorem for  rainbow $k$-edge-connected graphs as defined in the introduction. In this section we explain why the most natural analogue fails to hold.
 
Consider the following two properties in an edge-coloured directed graph $D$ and a pair of vertices $u,v\in D$.
\begin{enumerate}[(i)]
\item For any set of $k-1$ colours $S$, there is a rainbow $u$ to $v$ path $P$ avoiding colours in $S$.
\item There are $k$ edge-disjoint $u$ to $v$ paths $P_1, \dots, P_k$ such that $P_1\cup \dots\cup P_k$ is rainbow.
\end{enumerate}
The most natural analogue of Menger's Theorem for  rainbow $k$-connected graphs would say that for any graph we have (i) $\iff$ (ii). One reason this would be a natural analogue of Menger's Theorem is that there is fractional analogue of the statement (i) $\iff$ (ii). We say that a rainbow path $P$ contains a colour $c$ if $P$ has a colour $c$ edge.
\begin{proposition}\label{FractionalMenger}
Let $D$ be a edge-coloured directed graph, $u$ and $v$ two vertices in $D$, and $k$ a real number. The following are equivalent. 
\begin{enumerate}[(a)]
\item For any assignment of non-negative real number $y_c$ to every colour $c$, with $\sum_{c \text{ a colour}} y_c< k$, there is a rainbow $u$ to $v$ path $P$ with  $\sum_{c \text{ contained in  } P} y_c< 1$.
\item We can assign a non-negative real number $x_P$ to every rainbow $u$ to $v$ path $P$, such that for any colour $c$ we have $\sum_{P \text{ contains  } c} x_P\leq 1$ and also $\sum_{P \text{ a rainbow } u \text{ to }  v \text{ path}} x_P\geq k$.
\end{enumerate}
\end{proposition}
\begin{proof}
Let $k_a$ be the minimum of $\sum_{c \text{ a colour}} y_c$ over all choices of non-negative real numbers $y_c$ satisfying $\sum_{c \text{ contained in  } P} y_c\geq 1$ for all $u$ to $v$ paths $P$.
Similarly, we let $k_b$ be the maximum of $\sum_{P \text{ a rainbow } u \text{ to }  v \text{ path}} x_P$ over all choices of non-negative real numbers $x_P$ satisfying $\sum_{P \text{ contains  } c} x_P\leq 1$ for all colours $c$. 

It is easy to see that $k_a$ and $k_b$ are solutions of two linear programs which are dual to each other. Therefore, by the strong duality theorem (see~\cite{LinearProgrammingBook}) we have that $k_a=k_b$ which implies the proposition.
\end{proof}
The reason we say that Proposition~\ref{FractionalMenger} is an analogue of the statement ``(i) $\iff$ (ii)'' is that if the real numbers $y_c$ and $x_P$ were all in $\{0,1\}$ then (a) would be equivalent to (i) and (b) would be equivalent to (ii) (this is seen by letting $S=\{c: y_c=1\}$  and $\{P_1, \dots, P_k\}=\{P: x_P=1\}$).

Unfortunately (i) does not imply (ii) in a very strong sense. In fact even if (ii) was replaced by the weaker statement ``there are $k$ edge-disjoint rainbow $u$ to $v$ paths'', then (i) would still not imply (ii).
\begin{proposition}
For any $k$ there is a coloured directed graph $D_k$ with two vertices $u$ and $v$ such that the following hold.
\begin{enumerate}[(I)]
\item For any set of $k$ colours $S$, there is a rainbow $u$ to $v$ path $P$ avoiding colours in $S$.
\item Any pair $P_1$, $P_2$ of rainbow $u$ to $v$ paths have a common edge.
\end{enumerate}
\end{proposition}
\begin{proof}
We will construct a multigraph having the above property.  It is easy to modify the construction to obtain a simple graph. Fix $m>2k+1$.
The vertex set of $D$ is $\{x_0, \dots, x_m\}$ with $u=x_0$ and $v=x_m$. For each $i=0, \dots, m-1$, $D$ has $k+1$ copies of the edge $x_ix_{i+1}$ appearing with colours $i$, $m+1$, $m+2$, $\dots$, $m+k$. In other words $G$ is the union of $k+1$ copies of the path $x_0 x_1 \dots x_m$ one of which is rainbow, and the rest monochromatic.

Notice that $D$ satisfies (II). Indeed if $P_1$ and $P_2$ are $u$ to $v$ paths, then they must have vertex sequence $x_0 x_1 \dots x_m$. Since there are only $m+k$ colours in $D$ both $P_1$ and $P_2$ must have at least $m-k$ edges with colours from $\{0, \dots, m-1\}$. By the Pigeonhole Principle, since $2(m-k)>m$, there is some colour $i\in \{0, \dots, m\}$ such that both $P_1$ and $P_2$ have a colour $i$ edge. But the only colour $i$ edge in $D$ is $x_ix_{i+1}$ which must therefore be present in both $P_1$ and $P_2$.
\end{proof}

There is another, more subtle, reason why (i) does not imply (ii). Indeed if we had ``(i) $\implies$ (ii)'' then this would imply that every bipartite graph consisting of $n$ matchings of size $n$ contains a rainbow matching of size $n$. 

Indeed given a bipartite graph $G$ with bipartition $X \cup Y$ consisting of $n$ matchings of size $n$ construct an auxiliary graph $G'$ by adding two vertices $u$ and $v$ to $G$ with all edges from $u$ to $X$ and from $Y$ to $v$ present. These new edges all receive different colours which were not in $G$. It is easy to see for any set $S$ of $n-1$ colours, there is a rainbow $u$ to $v$ path in $G'$ i.e. (i) holds for this graph with $k=n$. In addition, for a set of paths $P_1, \dots, P_t$ with $P_1\cup \dots\cup P_t$ rainbow, it is easy to see that $\{P_1\cap E(G), \dots, P_t\cap E(G)\}$ is a rainbow matching in $G$ of size $t$. 

Therefore if ``(i) $\implies$ (ii)'' was true then we would have a rainbow matching in $G$ of size $n$. However, as noted in the introduction, there exist Latin squares without transversals, and hence bipartite graphs consisting of $n$ matchings of size $n$ containing no rainbow matching of size $n$.

The above discussion has hopefully convinced the reader that the natural analogue of Menger's Theorem for  rainbow $k$-connectedness is not true. Nevertheless, it would be interesting to see if any statements about connectedness carry over to  rainbow $k$-connected graphs.

\subsection*{Improving the bound in Theorem~\ref{MainTheorem}}
One natural open problem is to improve the dependency of $N_0$ on $\epsilon$ in Theorem~\ref{MainTheorem}. Throughout our proof we made no real attempt to do this. However there is one interesting modification which one can make in order to significantly improve the bound on $N_0$ which we mention here.

Notice that the directed graphs $D_{X'}$ in Section~\ref{MainTheorem} and the directed graph $D$ in Section~\ref{GoldenRatioTheorem} had one big difference in their definition---to define the graphs $D_{X'}$ we only considered edges starting in $X$, whereas to define the graph $D$, we considered edges starting from both $X_0$ and $Y_0$. It is possible to modify the proof of Theorem~\ref{MainTheorem} in order to deal with directed graphs closer to those we used in the proof of Theorem~\ref{GoldenRatioTheorem}. There are many nontrivial modifications which need to be made for this to work. However, the end result seems to be that the analogue of Lemma~\ref{IncrementLemma} only needs to be iterated $O(\log \epsilon_0^{-1})$ many times (rather than $O(\epsilon_0^{-1})$ as in the proof of Theorem~\ref{MainTheorem}). This would lead to an improved bound on $N$ in Theorem~\ref{MainTheorem} $N=O\left(\epsilon^{C\log\epsilon}\right)$ for some constant $C$. In the grand scheme of things, this is still a very small improvement to bound in Theorem~\ref{MainTheorem}, and so we do not include any further details here. It is likely that completely new ideas would be needed for a major improvement in the bound in Theorem~\ref{MainTheorem}.

\subsection*{Acknowledgement}
The author would like to thank J\'anos Bar\'at for introducing him to this problem as well as Ron Aharoni, Dennis Clemens, Julia Ehrenm\"uller, and Tibor Szab\'o for various discussions related to it. This research was supported by the Methods for Discrete Structures, Berlin graduate school (GRK 1408).

\bibliography{Rainbow}
\bibliographystyle{abbrv}
\end{document}